\documentclass[a4paper, 10pt, conference]{ieeeconf}      

\IEEEoverridecommandlockouts                              
\usepackage{amssymb}
\usepackage{amsmath}
\usepackage{graphicx}      
\usepackage{psfrag}

\newtheorem{defn}{Definition}
\newtheorem{rem}{Remark} 

\newtheorem{prop}{Proposotion}

\title{\LARGE \bf
Analysis and Comparison of Port-Hamiltonian Formulations for Field Theories - demonstrated by means of the Mindlin plate
}

\author{Markus Sch\"{o}berl and Andreas Siuka
\thanks{M. Sch\"{o}berl is an APART fellowship holder of the Austrian Academy of Sciences)}
\thanks{M. Sch\"{o}berl and A. Siuka are with the Institute of Automatic Control and Control Systems Technology,
        University of Linz, Altenberger Str. 69, 4040 Linz, Austria
        {\tt\small markus.schoeberl@jku.at, andreas.siuka@jku.at}}%
}

\begin{document}

\maketitle
\thispagestyle{empty}
\pagestyle{empty}

\begin{abstract}

This paper focuses on the port-Hamiltonian formulation of systems described by partial differential equations. Based
on a variational principle we derive the equations of motion as well as the boundary conditions in the well-known 
Lagrangian framework. Then it is of interest to reformulate the equations of motion in a port-Hamiltonian setting, where we
compare the approach based on Stokes-Dirac structures to a Hamiltonian setting that makes use of the involved bundle structure
similar to the one on which the variational approach is based. We will use the Mindlin plate, a distributed parameter system
with spatial domain of dimension two, as a running example. 

\end{abstract}

\section{Introduction}

Distributed parameter systems described by partial differential equations
arise in systems theory from a modeling and a control theoretic point
of view and are without doubt a challenging research problem, where
lot of progress has been achieved in the last years. Also the port-Hamiltonian
setting, originally developed in the finite dimensional scenario has
been transfered to infinite-dimensional systems, where e.g. the well-known
approach based on (Stokes-)Dirac structures (also known from the lumped
parameter scenario) is available, see e.g. \cite{Maschke2005,SchaftMaschke,Machelli2004I,Machelli2004II,MachelliSIAM,Machelli2005CDC}
and references therein.

Also in mathematical physics, systems described by partial differential
equations (pdes) are interpreted in a Hamiltonian setting, e.g. in
\cite{Giachetta,Gotay,Olver} and references therein, but in most
cases systems with trivial boundary conditions are considered, which
is not the case in many engineerings applications. Therefore, the
approach based on Stokes-Dirac structures has been setup to overcome
the problem of non-zero energy flow through the boundary. 

A different port-Hamiltonian approach is based on a bundle structure
with respect to independent and dependent coordinates (not necessarily
relying on an underlying Stokes-Dirac structure), see also \cite{Enns,SchoeberlMCMDSPiezo2008,SchlacherHam2008,SchoeberlMCMDS2010,SiukaActa,SchoeberlCDC,SchoeberlLHMNL}
which are all based on \cite{Olver} but adapted to control purposes,
i.e. modified in a sense, such that non-zero energy flow through the
boundary can be considered (such that boundary ports are included)
and furthermore control inputs on the domain and/or the boundary can
be included.

The main difference of the approach relying on Stokes-Dirac structures
and the approach using bundles, is that the Stokes-Dirac scenario
is based on the choice of proper energy variables (flows and efforts)
for which the power balance is formulated, whereas the second approach
is based on a given Hamiltonian density (the total energy density)
and the evaluation of the power balance is performed based on the
underlying bundle formalism in order to restructure the pdes such
that the energy flows are linked to the physics. This will have the
consequence, that the variational derivative is interpreted differently
and the choice of state variables is different, in the two mentioned
approaches.

The purpose of this paper is, that based on the well-known Lagrangian
setting for first order field theories i) the partial differential
equations and the boundary conditions derived using a variational
principle are reinterpreted using two different port-Hamiltonian settings,
which describe the same physical phenomenon but using a completely
different port-Hamiltonian representation, ii) by using the example
of a Mindlin plate all these concepts are visualized and compared
in great detail.

\section{Notation }

We will use differential geometric methods for our considerations
and the notation is similar to the one in \cite{Giachetta}, where
the interested reader can find much more details about this geometric
machinery. To keep the formulas short and readable we will use tensor
notation and especially Einstein's convention on sums. 

We use the standard symbol $\wedge$ for the exterior product (wedge
product), $\mathrm{d}$ is the exterior derivative, $\rfloor$ the
natural contraction between tensor fields. By $\partial_{\alpha}^{B}$
are meant the partial derivatives with respect to coordinates with
the indices $_{B}^{\alpha}$ and $[m^{\alpha\beta}]$ corresponds
to the matrix representation of the (second-order) tensor $m$ with
components $m^{\alpha\beta}$. E.g. taking a second-order tensor $m$
and a co-vector $\omega$, the components of the contraction $m\rfloor\omega$
read in local coordinates as $m^{\alpha\beta}\omega_{\alpha}$, where
the summation over $\alpha$ is performed (Einstein convention on
sums). 

Furthermore $C^{\infty}(\cdot)$ denotes the set of the smooth functions
on the corresponding manifold. Moreover we will not indicate the range
of the used indices when they are clear from the context. Additionally,
pull backs and pull back bundles are only stated when necessary, when
they follow from the context they are not indicated to avoid exaggerated
notation.

Let us consider the bundle $\mathcal{Y}\rightarrow\mathcal{D},\,(X^{A},y^{\alpha})\rightarrow(X^{A})$.
The first jet manifold $\mathcal{J}^{1}(\mathcal{Y})$ possesses the
coordinates $(X^{A},y^{\alpha},y_{A}^{\alpha})$, where the capital
Latin indices $A,B$ are used for the base manifold $\mathcal{D}$
(independent coordinates) and $y_{A}^{\alpha}$ denote derivative
coordinates of first order (derivatives of the dependent coordinates
with respect to the independent ones) as well as 
\[
\partial_{A}=\frac{\partial}{\partial X^{A}}\,,\,\partial_{\alpha}=\frac{\partial}{\partial y{}^{\alpha}}\,,\,\partial_{\alpha}^{A}=\frac{\partial}{\partial y{}_{A}^{\alpha}}.
\]
The jet structure also induces the so-called total derivative 
\[
d_{A}=\partial_{A}+y_{A}^{\alpha}\partial_{\alpha}+y_{AB}^{\alpha}\partial_{\alpha}^{B}
\]
 acting on elements including first order derivatives and $y_{AB}^{\alpha}$
correspond to derivative coordinates of second order living in $\mathcal{J}^{2}(\mathcal{Y})$,
the second jet manifold. Based on the bundle structure $\mathcal{Y}\rightarrow\mathcal{D}$
let us introduce the vertical tangent bundle $\mathcal{V}(\mathcal{Y})$,
as well as 
\[
\Lambda_{1}^{d}(\mathcal{Y})=\mathcal{T}^{*}(\mathcal{Y})\wedge(\overset{d}{\wedge}\mathcal{T}^{*}(\mathcal{D})),
\]
see also \cite{Giachetta}, with a typical element $\omega=\omega_{\alpha}\mathrm{d}y^{\alpha}\wedge\mathrm{d}V$
for $\Lambda_{1}^{d}(\mathcal{Y})$ where $\mathrm{d}V$ denotes the
volume element on the manifold $\mathcal{D}$, i.e. $\mathrm{d}V=\mathrm{d}X^{1}\wedge\ldots\mathrm{\wedge d}X^{d}$
with $\mathrm{dim}(\mathcal{D})=d$ and the functions $\omega_{\alpha}$
may depend on derivative coordinates. Furthermore, a typical element
for $\mathcal{V}(\mathcal{Y})$ reads as $v=v^{\alpha}\partial_{\alpha}$
and when $v^{\alpha}$ depends on derivative coordinates we call $v$
a generalized vertical vector field, see \cite{Olver}.

\section{Background material}

\subsection{Geometric preliminaries}

We will consider densities $\mathfrak{F}$ in the sequel (a quantity
that can be integrated), $\mathfrak{F}=\mathcal{F}\mathrm{d}V$ with
$\mathcal{F}\in C^{\infty}(\mathcal{J}^{1}(\mathcal{X}))$ (we restrict
ourselves to the first-order case). By $F=\int_{\mathcal{D}}\mathfrak{F}$
we denote the integrated quantity, where of course a section of the
bundle $\mathcal{Y}\rightarrow\mathcal{D}$, i.e. a map $y=\Phi(X)$
leading to $y_{A}=\partial_{A}\Phi(X)$ has to be plugged in to be
able to evaluate the integral properly. 
\begin{prop}
\label{prop:Given-the-density}Given the density $\mathfrak{F}=\mathcal{F}\mathrm{d}V$
and a generalized vertical vector field $v:\mathcal{D}\rightarrow\mathcal{V}(\mathcal{Y})$,
together with its first jet-prolongation $j^{1}(v)=v^{\alpha}\partial_{\alpha}+d_{A}(v^{\alpha})\partial_{\alpha}^{A}$,
see \cite{Olver,Giachetta}, we obtain the decomposition 
\begin{eqnarray}
\int_{\mathcal{D}}j^{1}(v)(\mathcal{F}\mathrm{d}V) & = & \int_{\mathcal{D}}v\rfloor\delta\mathfrak{F}+\int_{\partial\mathcal{D}}v\rfloor\delta^{\partial}\mathfrak{F}.\label{eq:LieDiffPde-1}
\end{eqnarray}
Here the map $\delta\mathfrak{F}=\delta{}_{\alpha}\mathcal{F}\,\mathrm{d}y^{\alpha}\wedge\mathrm{d}V$
(corresponding to the Euler Lagrange operator of $\mathfrak{F})$,
see \cite{Giachetta}, with the coefficients $\delta_{\alpha}\mathcal{\mathcal{F}}=\partial_{\alpha}\mathcal{F}-d_{A}\partial_{\alpha}^{A}\mathcal{F}$
(called the variational derivatives) is used, as well as the boundary
operator $\delta^{\partial}\mathfrak{F}=\partial_{\alpha}^{A}\mathcal{F}\,\mathrm{d}y^{\alpha}\wedge\mathrm{d}V_{A}$
with $\mathrm{d}V_{A}=\partial_{A}\rfloor\mathrm{d}V$ (the boundary
volume form).
\end{prop}

\begin{proof}
The proof follows by evaluating the Lie-derivative of the geometric
object $\mathfrak{F}$ with respect to the vector field $j^{1}(v)$
\begin{eqnarray}
j^{1}(v)(\mathcal{F}\mathrm{d}V) & = & \left(v^{\alpha}(\partial_{\alpha}\mathcal{F}-d_{A}\partial_{\alpha}^{A}\mathcal{F})+d_{A}(v^{\alpha}\partial_{\alpha}^{A}\mathcal{F})\right)\mathrm{d}V\nonumber \\
 & = & \left(v^{\alpha}\delta_{\alpha}\mathcal{F}+d_{A}(v^{\alpha}\partial_{\alpha}^{A}\mathcal{F})\right)\mathrm{d}V\label{eq:intbyparts}
\end{eqnarray}
and applying the Theorem of Stokes \cite{Olver} to (\ref{eq:intbyparts}).
\end{proof}
The relation (\ref{eq:LieDiffPde-1}) will be of key interest in the
forthcoming, since it provides a natural decomposition of the expression
$\int_{\mathcal{D}}j^{1}(v)(\mathcal{F}\mathrm{d}V)$ into a term
on the domain $\mathcal{D}$ and one on the boundary $\partial\mathcal{D}$.
Important is the case when the generalized vector-field $v$ is linked
to the solution of a pde system (via its semi-group, that $v$ may
generate), then the formal change of $F=\int_{\mathcal{D}}\mathcal{F}\mathrm{d}V$
along solutions of a pde system can be computed as $\int_{\mathcal{D}}j^{1}(v)(\mathcal{F}\mathrm{d}V)$
(provided all operations are admissible), which we denote by $\dot{F}=\int_{\mathcal{D}}j^{1}(v)(\mathcal{F}\mathrm{d}V)$
in this special case.

\subsection{Dirac structures}

Based on the space of power variables $\mathcal{F}\times\mathcal{E}$
(flows and efforts) and the symmetric bilinear pairing
\begin{equation}
\ll(f_{1},e_{1}),(f_{2},e_{2})\gg:=\left\langle e_{1},f_{2}\right\rangle +\left\langle e_{2},f_{1}\right\rangle \label{eq:DiracPairing}
\end{equation}
where $\left\langle \cdot,\cdot\right\rangle $ is the dual product
of the linear spaces $\mathcal{F}$ and $\mathcal{E}=\mathcal{F}^{*}$
a Dirac structure is a linear subspace $\mathbb{D}\subset\mathcal{F}\times\mathcal{E}$
such that $\mathbb{D}=\mathbb{D}^{\bot}$ with respect to the pairing
(\ref{eq:DiracPairing}). For $(f,e)\in\mathbb{D}$ one has $\left\langle e,f\right\rangle =0$
such that the Dirac structure preserves power. This concept can be
transfered to the case where $\mathcal{F}$ and $\mathcal{E}$ are
spaces of vector-valued functions over a spatial domain $\mathcal{D}$,
then infinite dimensional systems are the focus, and to allow for
non-zero energy flow through the boundary the so-called Stokes-Dirac
structure is introduced, see \cite{Maschke2005,SchaftMaschke,Machelli2004I,Machelli2005CDC}
and section \ref{sub:Approach-based-on}.

\section{Lagrangian Framework}

In this section we recapitulate the well-known Lagrangian framework
for first-order field theories, and we will derive the partial differential
equations as well as the boundary conditions in a geometric fashion.
Thus, we consider a bundle 
\[
\mathcal{Q}\rightarrow\mathcal{D}_{L},\,(q^{\alpha},t^{0},X^{A})\rightarrow(t^{0},X^{A})
\]
where we use the shortcut $x^{i}=(t^{0},X^{A})$ such that the independent
variables are the time $t^{0}$ and the spatial ones $X^{A}\,,\, A=1,\ldots\dim(\mathcal{D}_{L})-1=n$
. A first order Lagrangian takes the form

\begin{equation}
\mathfrak{L}:\mathcal{J}^{1}(Q)\rightarrow\overset{n+1}{\wedge}\mathcal{T}^{*}(\mathcal{D}_{L})\label{eq:}
\end{equation}
$\mathfrak{L}=\mathcal{L}\omega$ with $\mathcal{L}\in C^{\infty}\left(\mathcal{J}^{1}\left(Q\right)\right)$
together with the volume element $\omega$ that meets 
\[
\omega=\mathrm{d}t^{0}\wedge\mathrm{d}X^{1}\ldots\wedge\mathrm{d}X^{n}\,,\,\,\,\omega_{i}=\partial_{i}\rfloor\omega\,,\,\,\,\, i=0,\ldots,n
\]
The variational problem for a section $s:\mathcal{D}_{L}\mathcal{\rightarrow Q}$
is the following%
\footnote{At this point the pull back is essential and therefore indicated,
i.e. $j^{1}(\psi_{\epsilon}\circ s)^{*}\mathfrak{L}$ means, that
the first prolongation of $(\psi_{\epsilon}\circ s)$ has to be plugged
in into $\mathfrak{L}=\mathcal{L}\omega$ in order to evaluate the
integral properly.%
}
\begin{equation}
\left.\left(d_{\epsilon}\int_{\mathcal{D}_{L}}j^{1}(\psi_{\epsilon}\circ s)^{*}\mathfrak{L}\right)\right|_{\epsilon=0}=0,\label{eq:VarProb}
\end{equation}
where the flow $\psi_{\epsilon}$ is used to deform sections $s:\mathcal{D}_{L}\mathcal{\rightarrow Q}$
and whose generator is a vertical vector field $v_{L}:\mathcal{Q}\rightarrow\mathcal{V}(Q)$. 

This is a well-known problem and treated for example in \cite{Giachetta,Olver}
and references therein. It is obvious that (\ref{eq:VarProb}) is
equivalent to 
\begin{eqnarray}
\int_{\mathcal{D}_{L}}(j^{2}s)^{*}\left(j^{1}(v_{L})(\mathfrak{L})\right) & = & 0,\label{eq:VarProb1}
\end{eqnarray}
see for example \cite{Giachetta} and based on (\ref{eq:LieDiffPde-1})
where we replace $\mathcal{F}$ be $\mathcal{L}$ we obtain the decomposition
\begin{equation}
\int_{\mathcal{D}_{L}}v_{L}\rfloor\delta\mathfrak{L}+\int_{\partial\mathcal{D}_{L}}v_{L}\rfloor\delta^{\partial}\mathfrak{L}=0.\label{eq:VarSplitB}
\end{equation}
Consequently, the partial differential equations for a first order
Lagrangian follow as 
\begin{equation}
\delta_{\alpha}(\mathcal{L})=0\,,\,\,\,\delta_{\alpha}=\partial_{\alpha}-d_{i}\partial_{\alpha}^{i}\label{eq:LAGEUL}
\end{equation}
and the boundary term is the second term in (\ref{eq:VarSplitB})
and reads in local coordinates as 
\begin{equation}
\int_{\partial\mathcal{D}_{L}}v_{L}^{\alpha}\partial_{\alpha}^{i}\mathcal{L}\omega_{i}=0.\label{eq:BondGen}
\end{equation}
The boundary conditions can be fulfilled by either allowing for no
variations on (a part of) $\partial\mathcal{D}_{L}$, i.e. $v_{L}^{\alpha}=0$
or by $\partial_{\alpha}^{i}\mathcal{L}=0$ or by a combinations of
both approaches. Possible is also the inclusion of external boundary
variables $F_{e,\alpha}^{i}$ such that 
\begin{equation}
\int_{\partial\mathcal{D}_{L}}v_{L}^{\alpha}(\partial_{\alpha}^{i}\mathcal{L}-F_{e,\alpha}^{i})\omega_{i}=0.\label{eq:BondGen-1}
\end{equation}
has to be met.

\section{Port-Hamiltonian Picture}

Now we turn to the Hamiltonian picture, where we discuss two different
port-Hamiltonian formulations. We will restrict ourselves to systems
without dissipation and without distributed control for simplicity,
but these properties can be included in both formalisms in a straightforward
manner, see \cite{SchaftMaschke,SchoeberlMCMDSPiezo2008}.

\subsection{Geometric approach based on underlying bundle structure\label{sub:Geometric-approach-based}}

We will introduce port-Hamiltonian systems described by pdes based
on a power balance relation, such that the power balance relation
together with the structure of the equations represent the physical
process. 
\begin{defn}
\label{A-port-controlled}A port-Hamiltonian boundary control system
without dissipation on a bundle $\mathcal{X}\rightarrow\mathcal{D_{H}},\,(x^{\alpha},X^{A})\rightarrow(X^{A})$
takes the form of 
\begin{equation}
\begin{array}{ccl}
\dot{x} & = & \mathcal{J}(\mathrm{\delta}\mathfrak{H)}\end{array}\label{eq:HamTimeInvStandardCF-1}
\end{equation}
with the Hamiltonian $\mathfrak{H}=\mathcal{H}\Omega\,,\,\,\Omega=\mathrm{d}X^{1}\wedge\ldots\wedge\mathrm{d}X^{d}$,
where $\mathcal{H}\in C^{\infty}(\mathcal{J}^{1}(\mathcal{X}))$ and
additional boundary conditions (possibly including boundary inputs,
optionally leading to so-called boundary ports). The map $\mathcal{J}$
is of the form $\mathbb{\mathcal{J}}:\Lambda_{1}^{d}(\mathcal{X})\rightarrow\mathcal{V\textrm{(}\mathcal{X}\textrm{)}}$
where $\mathcal{J}$ is a skew-symmetric map. 
\end{defn}
In general the map $\mathcal{J}$ can be a differential operator,
see our paper \cite{SchoeberlLHMNL}, but within this contribution
we exclude this case (since in many examples, e.g. mechanics this
is not required). 

Now we make use of proposition \ref{prop:Given-the-density} and replace
$\mathfrak{F}$ by $\mathfrak{H}$ in (\ref{eq:LieDiffPde-1}). Setting
$v=\dot{x}$ we obtain 
\begin{eqnarray}
\dot{H} & = & \int_{\partial\mathcal{D_{H}}}\dot{x}\rfloor\delta^{\partial}\mathfrak{H}=\int_{\partial\mathcal{D_{H}}}\dot{x}^{\alpha}\partial_{\alpha}^{A}\mathcal{H}\Omega_{A}\label{eq:HdotPdeoO}
\end{eqnarray}
where $\Omega_{A}=\partial_{A}\rfloor\Omega$, which reflects the
power balance, since the total change of the functional $H$ along
solutions of (\ref{eq:HamTimeInvStandardCF-1}), is affected by a
boundary port (if it exists) depending on the boundary conditions.
See e.g. \cite{SchoeberlMCMDSPiezo2008} for a formal introduction
concerning the boundary ports.

\subsection{Approach based on underlying Stokes-Dirac Structure\label{sub:Approach-based-on}}

Following \cite{Machelli2005CDC} we shortly recapitulate the port-Hamiltonian
framework based on Stokes-Dirac structures. For more details we refer
to \cite{Maschke2005,SchaftMaschke}.

\begin{flushleft}
We consider the space of flows $\mathcal{F}$ and the space of efforts
$\mathcal{E}$, which are spaces of vector-valued functions over a
spatial domain $\mathcal{D}$. Given $\mathcal{J}_{SD}$ a skew-adjoint
matrix differential operator \cite{Machelli2005CDC}, the space
\[
\mathbb{D}=\left\{ \left.\left(f,e,w\right)\in\mathcal{F}\times\mathcal{E}\times\mathcal{W}\right|f=-\mathcal{J}_{SD}e,w=\mathcal{B}_{\mathcal{D}}(e)\right\} 
\]
is a Stokes-Dirac structure, regarding the pairing
\begin{align*}
\begin{array}{c}
\!\!\!\!\!\!\!\!\!\!\!\!\!\!\!\!\!\!\!\!\!\!\!\!\!\!\!\!\!\!\!\!\!\!\!\!\!\!\!\!\ll\left(f_{1},e_{1},w_{1}\right),\left(f_{2},e_{2},w_{2}\right)\gg=\\
\qquad=\int_{\mathcal{D}}[e_{1}^{T}f_{2}+e_{2}^{T}f_{1}]\mathrm{d}V+\int_{\partial\mathcal{D}}\mathcal{B}_{\mathcal{J}_{SD}}(w_{1},w_{2})\mathrm{d}A
\end{array}
\end{align*}
where $\mathcal{B}_{\mathcal{J}_{SD}}$ is a boundary differential
operator induced by $\mathcal{J}_{SD}$ any by slight abuse of notation
$\mathrm{d}A$ corresponds to the boundary volume element. The map
$\mathcal{B}_{\mathcal{D}}$ is a boundary operator and the boundary
variables are $w$.
\par\end{flushleft}

If $(f,e,w)\in\mathbb{D}$ then 
\begin{eqnarray}
0 & = & \int_{\mathcal{D}}e^{T}f\mathrm{d}V+\frac{1}{2}\int_{\partial\mathcal{D}}\mathcal{B}_{\mathcal{J}_{SD}}(w,w)\mathrm{d}A\label{eq:DiracEnergy}
\end{eqnarray}
holds. Given an energy density $\mathcal{H}\mathrm{d}V$ where $\mathcal{H}$
depends on the energy variables, a port-Hamiltonian boundary control
system without dissipation can be stated as 
\begin{equation}
f=-\mathcal{J}_{SD}e\,,\,\, w=\mathcal{B}_{\mathcal{D}}(e)\label{eq:pdeDirac}
\end{equation}
where the energy variables $f$, are linked to the state variables
$\chi$ via $f=-\dot{\chi}$ and the efforts variables follow from
$e=\partial_{\chi}\mathcal{H}$.
\begin{rem}
Originally, in \cite{Machelli2005CDC} instead of $e=\partial_{\chi}\mathcal{H}$
the authors use $e=\delta_{\chi}H$ where $H=\int_{\mathcal{D}}\mathcal{H}\mathrm{d}V$,
but since $\mathcal{H}$ depends on energy variables, the variational
derivative degenerates to a 'partial' one. 
\end{rem}
Furthermore, from $H=\int_{\mathcal{D}}\mathcal{H}\mathrm{d}V$ and
the relations (\ref{eq:DiracEnergy}) and (\ref{eq:pdeDirac}) one
has
\begin{equation}
\dot{H}=\frac{1}{2}\int_{\partial\mathcal{D}}\mathcal{B}_{\mathcal{J}_{SD}}(w,w)\mathrm{d}A\label{eq:BalanceDirac}
\end{equation}
since $\dot{H}=\int_{\mathcal{D}}\partial_{\chi}\mathcal{H}\dot{\xi}\mathrm{d}V=-\int_{\mathcal{D}}e^{T}f\mathrm{d}V$.
\begin{rem}
To derive this energy balance also proposition \ref{prop:Given-the-density}
can be applied, but since no jet-variables are included it simplifies
to $\dot{H}=\int_{\mathcal{D}}j^{1}(v)(\mathcal{H}\mathrm{d}V)=\int_{\mathcal{D}}v^{\alpha}\partial_{\alpha}\mathcal{H}\mathrm{d}V$
and to derive (\ref{eq:BalanceDirac}) a further integration by parts
must be performed, since $v$ corresponds to $-f$ which involves
the differential operator $\mathcal{J}_{SD}$.
\end{rem}

\section{The Mindlin Plate}

Let us consider a rectangular plate with lengths $l_{x},l_{y}$, where
$h$ will denote the thickness, which will be modeled based on the
hypothesis stated by Mindlin. Therefore, we choose as independent
coordinates the vertical deflection $w$ of the mid-plane as well
as the rotations of a transverse normal to the $X$ and $Y$ direction
termed $\psi$ and $\phi,$ respectively. The kinetic energy density
$\mathcal{K}$ and the potential energy density $\mathcal{V}$ can
be stated as 

\begin{eqnarray*}
\mathcal{K} & = & \frac{\rho}{2}(\frac{h^{3}}{12}(\psi_{t}^{2}+\phi_{t}^{2})+hw_{t}^{2})
\end{eqnarray*}
and

\begin{eqnarray*}
\mathcal{V} & = & \frac{1}{2}kGh\left[(w_{X}-\psi)^{2}+(w_{Y}-\phi)^{2}\right]\\
 &  & +\frac{1}{2}D\frac{1-\nu}{2}(\psi_{Y}+\phi_{X})^{2}\\
 &  & +\frac{1}{2}(D(\psi_{X}^{2}+\nu\phi_{Y}\psi_{X})+D(\phi_{Y}^{2}+\nu\phi_{Y}\psi_{X})),
\end{eqnarray*}
where $\nu$ is the Poisson ratio, $k=\frac{\pi^{2}}{12},$ and $G,D$
are the plate stiffness and the plate module, respectively, see \cite{Machelli2005CDC}
and references therein.
\begin{rem}
The subscripts $t,X,Y$ correspond to the derivatives with respect
to these independent variables, according to the jet-bundle structure
in the Lagrangian framework. The subscripts $x,y$ to be used later,
correspond to quantities which are connected to the spatial variables
$X$ and $Y$ but they must not be confused with derivative variables.
Since we are in a time-invariant setting, we will use later on also
the $\dot{}$ notation, for time derivatives, instead of the subscripts
$t$.
\end{rem}
To derive the equations of motion we will use the variational principle
in a Lagrangian setting. Then given the partial differential equations,
we will interpret them in a Hamiltonian setting, either using the
approach presented in section \ref{sub:Geometric-approach-based}
and using an approach based on the Stokes-Dirac structure as in section
\ref{sub:Approach-based-on}.

\subsection{The Lagrangian picture}

In the Lagrangian framework we consider the bundle 
\begin{equation}
\mathcal{Q}\rightarrow\mathcal{D}_{\mathcal{L}},\,(w,\psi,\phi,t,X,Y)\rightarrow(t,X,Y)\label{eq:BundleLag}
\end{equation}
together with the Lagrangian density $\mathfrak{L}=\mathcal{L}\omega$
with $\mathcal{L}=\mathcal{K}-\mathcal{V},\,\omega=\mathrm{d}t\wedge\mathrm{d}X\wedge\mathrm{d}Y.$
The variational derivatives follow form the chosen bundle structure
(\ref{eq:BundleLag}) and follow to
\begin{eqnarray*}
\delta_{w} & = & \partial_{w}-d_{t}\partial_{w}^{t}-d_{X}\partial_{w}^{X}-d_{Y}\partial_{w}^{Y}\\
\delta_{\psi} & = & \partial_{\psi}-d_{t}\partial_{\psi}^{t}-d_{X}\partial_{\psi}^{X}-d_{Y}\partial_{\psi}^{Y}\\
\delta_{\phi} & = & \partial_{\phi}-d_{t}\partial_{\phi}^{t}-d_{X}\partial_{\phi}^{X}-d_{Y}\partial_{\phi}^{Y}.
\end{eqnarray*}
From $\delta_{w}\mathcal{L}=0,\,\delta_{\psi}\mathcal{L}=0,\,\delta_{\phi}\mathcal{L}=0$
corresponding to (\ref{eq:LAGEUL}) we derive the partial differential
equations
\begin{eqnarray*}
\rho hw_{tt} & = & kGh(w_{XX}-\psi_{X})+kGh(w_{YY}-\phi_{Y})\\
\rho\frac{h^{3}}{12}\psi_{tt} & = & kGh(w_{X}-\psi)+D(\psi_{XX}+\nu\phi_{XY})\\
 &  & +\frac{1}{2}D(1-\nu)(\psi_{YY}+\phi_{XY})\\
\rho\frac{h^{3}}{12}\phi_{tt} & = & kGh(w_{Y}-\phi)+D(\phi_{YY}+\nu\psi_{XY})\\
 &  & +\frac{1}{2}D(1-\nu)(\psi_{XY}+\phi_{XX}).
\end{eqnarray*}
If we introduce
\begin{eqnarray}
M_{x} & = & D(\psi_{X}+\nu\phi_{Y})=-\partial_{\psi}^{X}\mathcal{L}\nonumber \\
M_{y} & = & D(\phi_{Y}+\nu\psi_{X})=-\partial_{\phi}^{Y}\mathcal{L}\nonumber \\
M_{xy} & = & D\frac{1-\nu}{2}(\psi_{Y}+\phi_{X})=-\partial_{\psi}^{Y}\mathcal{L=}-\partial_{\phi}^{X}\mathcal{L}\label{eq:MQ}\\
Q_{x} & = & kGh(w_{X}-\psi)=-\partial_{w}^{X}\mathcal{L}\nonumber \\
Q_{y} & = & kGh(w_{Y}-\phi)=-\partial_{w}^{Y}\mathcal{L}\nonumber 
\end{eqnarray}
then the equations of motion take the familiar form%
\footnote{In order to be comparable with the literature we sometimes use $\partial_{X}$and
$\partial_{Y}$ although in a strict mathematical sense it should
be $d_{X}$ and $d_{Y}.$%
}
\begin{eqnarray}
\rho hw_{tt} & = & \partial_{X}Q_{x}+\partial_{Y}Q_{y}\nonumber \\
\rho\frac{h^{3}}{12}\psi_{tt} & = & Q_{x}+\partial_{X}M_{x}+\partial_{Y}M_{xy}\label{eq:PDELagMind}\\
\rho\frac{h^{3}}{12}\phi_{tt} & = & Q_{y}+\partial_{Y}M_{y}+\partial_{X}M_{xy}.\nonumber 
\end{eqnarray}
The boundary conditions follow from 
\[
\int_{\partial\mathcal{D_{L}}}\left(w_{L}\partial_{w}^{i}\mathcal{L}+\psi_{L}\partial_{\psi}^{i}\mathcal{L}+\phi_{L}\partial_{\phi}^{i}\mathcal{L}\right)\partial_{i}\rfloor(\mathrm{d}t\wedge\mathrm{d}X\wedge\mathrm{d}Y)=0
\]
where the variational vector field $v_{L}$ takes the form $v_{L}=w_{L}\partial_{w}+\psi_{L}\partial_{\psi}+\phi_{L}\partial_{\phi}$
and on the time-boundary no variation takes place (i.e. when $i=0$
then $v_{L}=0$) and (\ref{eq:MQ}) has to be used. Therefore we have
\begin{eqnarray}
\int_{\partial\mathcal{D_{L}}}\left(w_{L}Q_{x}+\psi_{L}M_{x}+\phi_{L}M_{xy}\right)\mathrm{d}t\wedge\mathrm{d}Y\nonumber \\
-\int_{\partial\mathcal{D_{L}}}\left(w_{L}Q_{y}+\psi_{L}M_{xy}+\phi_{L}M_{y}\right)\mathrm{d}t\wedge\mathrm{d}X=0\label{eq:LAGBondMind}
\end{eqnarray}
such that, e.g. if at $X=0$ we have that $w_{L}$ is arbitrary, then
$Q_{x}$ has to vanish or has to be compensated by an external boundary
term as in (\ref{eq:BondGen-1}), such that the familiar boundary
conditions are recovered.

\subsection{The Hamiltonian picture}

Based on the partial differential equations (\ref{eq:PDELagMind})
and the boundary conditions (\ref{eq:LAGBondMind}) we discuss the
two presented port-Hamiltonian formulations as well as the power balance
relations corresponding to the particular representation.

\subsubsection{Geometric approach}

Now we consider the bundle (which is different form the one in the
Lagrangian setting) 
\[
\mathcal{X}\rightarrow\mathcal{D_{H}},\,(w,\psi,\phi,p_{w},p_{\psi},p_{\phi},X,Y)\rightarrow(X,Y).
\]
From the Legendre transform we derive the temporal momenta $p_{w}=\partial_{w}^{t}\mathcal{L},\, p_{\psi}=\partial_{\psi}^{t}\mathcal{L},\, p_{\phi}=\partial_{\phi}^{t}\mathcal{L}$
which read as
\begin{equation}
p_{w}=\rho hw_{t}\,,\,\, p_{\psi}=\rho\frac{h^{3}}{12}\psi_{t}\,,\,\, p_{\phi}=\rho\frac{h^{3}}{12}\phi_{t}\label{eq:Legendre}
\end{equation}
and the Hamiltonian follows as 
\[
\mathcal{H}=\dot{w}p_{w}+\dot{\psi}p_{\psi}+\dot{\phi}p_{\phi}-\mathcal{L}.
\]
In the coordinates $(w,\psi,\phi,p_{w},p_{\psi},p_{\phi})$ together
with (\ref{eq:Legendre}) one has $\mathcal{H}=\mathcal{K}+\mathcal{V}$. 

To derive the Hamiltonian formulation as in (\ref{eq:HamTimeInvStandardCF-1})
we set $x=(w,\psi,\phi,p_{w},p_{\psi},p_{\phi})$ and obtain 
\begin{eqnarray}
\dot{x} & = & \mathcal{J}(\delta\mathfrak{H})\label{eq:pdeGeoEx}
\end{eqnarray}
which reads as
\[
\left[\begin{array}{c}
\dot{w}\\
\dot{\psi}\\
\dot{\phi}\\
\dot{p}_{w}\\
\dot{p}_{\psi}\\
\dot{p}_{\phi}
\end{array}\right]=\left[\begin{array}{cccccc}
0 & 0 & 0 & 1 & 0 & 0\\
0 & 0 & 0 & 0 & 1 & 0\\
0 & 0 & 0 & 0 & 0 & 1\\
-1 & 0 & 0 & 0 & 0 & 0\\
0 & -1 & 0 & 0 & 0 & 0\\
0 & 0 & -1 & 0 & 0 & 0
\end{array}\right]\left[\begin{array}{c}
\delta_{w}\mathcal{H}\\
\delta_{\psi}\mathcal{H}\\
\delta_{\phi}\mathcal{H}\\
\delta_{p_{w}}\mathcal{H}\\
\delta_{p_{\psi}}\mathcal{H}\\
\delta_{p_{\phi}}\mathcal{H}
\end{array}\right].
\]
The variational derivatives in this setting take the form
\begin{eqnarray*}
\delta_{w} & = & \partial_{w}-d_{X}\partial_{w}^{X}-d_{Y}\partial_{w}^{Y}\\
\delta_{\psi} & = & \partial_{\psi}-d_{X}\partial_{\psi}^{X}-d_{Y}\partial_{\psi}^{Y}\\
\delta_{\phi} & = & \partial_{\phi}-d_{X}\partial_{\phi}^{X}-d_{Y}\partial_{\phi}^{Y}
\end{eqnarray*}
and $\delta_{p_{w}}=\partial_{p_{w}}\,,\,\delta_{p_{\psi}}=\partial_{p_{\psi}}\,,\,\delta_{p_{\phi}}=\partial_{p_{\phi}}$because
of the different bundle structure compared to the Lagrangian approach. 

The boundary ports follow from (\ref{eq:HdotPdeoO}) and we obtain
\begin{eqnarray*}
\dot{H} & = & \int_{\partial\mathcal{D}_{H}}\dot{x}^{\alpha}\partial_{\alpha}^{A}\mathcal{H}\partial_{A}\rfloor(\mathrm{d}X\wedge\mathrm{d}Y)\\
 & = & \int_{\partial\mathcal{D_{H}}}(\dot{x}^{\alpha}\partial_{\alpha}^{X}\mathcal{H}\mathrm{d}Y-\dot{x}^{\alpha}\partial_{\alpha}^{Y}\mathcal{H}\mathrm{d}X).
\end{eqnarray*}
From the special choice of the Hamiltonian $\mathcal{H}$ we observe
that the expressions $\partial_{\alpha}^{A}\mathcal{L}$ and $\partial_{\alpha}^{A}\mathcal{H}$
correspond (apart form the sign), and therefore we also have 
\begin{eqnarray*}
\partial_{w}^{X}\mathcal{H}=Q_{x} & \,,\, & \partial_{w}^{Y}\mathcal{H}=Q_{y}\\
\partial_{\psi}^{X}\mathcal{H}=M_{x} & \,,\, & \partial_{\psi}^{Y}\mathcal{H}=M_{xy}\\
\partial_{\phi}^{X}\mathcal{H}=M_{xy} & \,,\, & \partial_{\phi}^{Y}\mathcal{H}=M_{y}
\end{eqnarray*}
which consequently leads to the power balance relation
\begin{eqnarray}
\dot{H} & = & \int_{\partial\mathcal{D_{H}}}(\dot{w}Q_{x}+\dot{\psi}M_{x}+\dot{\phi}M_{xy})\mathrm{d}Y\nonumber \\
 &  & -\int_{\partial\mathcal{D_{H}}}(\dot{w}Q_{y}+\dot{\psi}M_{yx}+\dot{\phi}M_{y})\mathrm{d}X\label{eq:HdotMinGeo}
\end{eqnarray}
which is based on (\ref{eq:HdotPdeoO}). 
\begin{rem}
It should be noted that depending on the boundary conditions, along
$\partial\mathcal{D}$ a boundary port appears only if in the pairings
$\dot{x}^{\alpha}\partial_{\alpha}^{A}\mathcal{H}$ both 'players'
are not equal to zero. Furthermore it should be noted that in $\dot{x}^{\alpha}$
only $(\dot{w},\dot{\psi},\dot{\phi})$ remain, since in $\mathcal{H}$
only jet variables with respect to $w,\psi,\phi$ appear, i.e. there
is a $w_{X}$ present but no $(p_{w})_{X}$ and so on. 
\end{rem}

\subsubsection{The Stokes-Dirac approach}

This approach is not based on a bundle structure, which distinguishes
dependent and independent variables strictly, but uses so-called energy
variables instead. Therefore, let us introduce the strain variables
as \cite{Machelli2005CDC} 
\begin{eqnarray}
\Gamma_{x} & = & -\psi_{X}\nonumber \\
\Gamma_{y} & = & -\phi_{Y}\nonumber \\
\Gamma_{xy} & = & -(\psi_{Y}+\phi_{X})\label{eq:Strains}\\
\Gamma_{xz} & = & w_{X}-\psi\nonumber \\
\Gamma_{yz} & = & w_{Y}-\phi.\nonumber 
\end{eqnarray}
Then one can introduce as state $\chi$ which consists of the momentum
variables and the strains 
\[
\chi=(\rho h\dot{w},\Gamma_{xz},\Gamma_{yz},\rho\frac{h^{3}}{12}\dot{\psi},\rho\frac{h^{3}}{12}\dot{\phi},\Gamma_{x},\Gamma_{y},\Gamma_{xy})
\]
and the Hamiltonian $\mathcal{H}=\mathcal{K}+\mathcal{V}$ can be
rewritten as
\begin{eqnarray*}
\mathcal{K} & = & \frac{\rho}{2}(\frac{h^{3}}{12}(\dot{\psi}^{2}+\dot{\phi}^{2})+h\dot{w}^{2})
\end{eqnarray*}
and
\begin{eqnarray*}
\mathcal{V} & = & \frac{1}{2}(Q_{x}\Gamma_{xz}+Q_{y}\Gamma_{yz}-M_{xy}\Gamma_{xy}-M_{x}\Gamma_{x}-M_{y}\Gamma_{y})
\end{eqnarray*}
such that $e=\partial_{\chi}\mathcal{H}$ follows to 
\[
e=(\dot{w},Q_{x},Q_{y},\dot{\psi},\dot{\phi},-M_{x},-M_{y},-M_{xy}).
\]
The partial differential equations can be stated as 
\begin{equation}
\dot{\chi}=\mathcal{J}_{SD}e=\mathcal{J}_{SD}\partial_{\chi}\mathcal{H}\label{eq:pdeDiracEx}
\end{equation}
where $\mathcal{J}_{SD}$ takes the form
\[
\left[\begin{array}{cccccccc}
0 & \partial_{X} & \partial_{Y} & 0 & 0 & 0 & 0 & 0\\
\partial_{X} & 0 & 0 & -1 & 0 & 0 & 0 & 0\\
\partial_{Y} & 0 & 0 & 0 & -1 & 0 & 0 & 0\\
0 & 1 & 0 & 0 & 0 & -\partial_{X} & 0 & -\partial_{Y}\\
0 & 0 & 1 & 0 & 0 & 0 & -\partial_{Y} & -\partial_{X}\\
0 & 0 & 0 & -\partial_{X} & 0 & 0 & 0 & 0\\
0 & 0 & 0 & 0 & -\partial_{Y} & 0 & 0 & 0\\
0 & 0 & 0 & -\partial_{Y} & -\partial_{X} & 0 & 0 & 0
\end{array}\right].
\]
The energy balance follows from equation (\ref{eq:BalanceDirac})
where (\ref{eq:DiracEnergy}) has to be evaluated by a further integration
by parts, since $\mathcal{J}_{SD}$ is a differential operator which
contributes to the boundary expression, i.e. $\mathcal{B}_{\mathcal{J}_{SD}}$
has to be constructed, see \cite{Machelli2005CDC}.

\section{Discussion and Comparison}

In this section we will discuss the main differences of the two presented
port-Hamiltonian scenarios as in sections (\ref{sub:Geometric-approach-based})
and (\ref{sub:Approach-based-on}) where we will highlight these aspects
by focusing on the presented example, the Mindlin plate.

\subsection{State variables}

The state variables in the geometric approach consist of the displacements/deflections
and the temporal momenta, i.e. $x=(w,\psi,\phi,p_{w},p_{\psi},p_{\phi})$
in our example where in many cases (mechanical systems) the temporal
momenta can be derived from a given Lagrangian by means of the Legendre
transformation. These temporal momenta are introduced mainly to obtain
explicit partial differential equations where the state variables
are differentiated with respect to a curve parameter, which is the
time. 

In contrast to this, in the approach based on the Stokes-Dirac structures,
energy variables are used, such that the strains are introduced in
mechanical applications and one has $\chi=(\rho h\dot{w},\Gamma_{xz},\Gamma_{yz},\rho\frac{h^{3}}{12}\dot{\psi},\rho\frac{h^{3}}{12}\dot{\phi},\Gamma_{x},\Gamma_{y},\Gamma_{xy})$.
From a conceptional point of view the use of energy/power variables
may be beneficial since they are linked to the power balance relation
in a simple manner, but as can be seen already using the Mindlin plate
example, the five strain variables have to be derived from the three
independent deflection/displacement variables $(w,\psi,\phi)$ by
differentiation, see (\ref{eq:Strains}), such that additionally to
the partial differential equations (\ref{eq:pdeDiracEx}) also the
compatibility conditions (\ref{eq:Strains}) must be listed, such
that a constrained Hamiltonian representation is apparent.

\subsection{Control issues}

From a control point of view control methods like damping injection
or control by interconnection can be performed equivalently using
the two presented Hamiltonian representations, see for example \cite{SiukaActa,MachelliSIAM}
where the Timoshenko beam is analyzed. However as stated also above,
the use of energy variables allows for controlling the system for
instance to zero strain configuration, but the global position in
space cannot be controlled in a straightforward manner since the deflection/displacement
coordinates do not enter the formalism, in contrast to the approach
as in (\ref{sub:Geometric-approach-based}).

\subsection{The skew-symmetric operators $\mathcal{J}$ and $\mathcal{J}_{SD}$ }

By inspection it becomes apparent that $\mathcal{J}$ and $\mathcal{J}_{SD}$
differ significantly, since $\mathcal{J}$ is no differential operator
in contrast to $\mathcal{J}_{SD}$. This also has severe consequences
for the expressions $\delta\mathfrak{H}$ as in (\ref{eq:HamTimeInvStandardCF-1})
and $e=\partial_{\chi}\mathcal{H}$ as in (\ref{eq:pdeDirac}), where
it is vice versa, i.e. $\delta$ is a variational derivative and in
the Stokes-Dirac approach a partial derivative appears. Let us consider
for instance the fourth equation of (\ref{eq:pdeDiracEx}) which reads
as
\[
\rho\frac{h^{3}}{12}\ddot{\psi}=1\cdot Q_{x}-\partial_{X}(-M_{x})-\partial_{Y}(-M_{xy})
\]
as well as the fifth equation of (\ref{eq:pdeGeoEx}) which is
\[
\dot{p}_{\psi}=-\delta_{\psi}\mathcal{H}=-(\partial_{\psi}-d_{X}\partial_{\psi}^{X}-d_{Y}\partial_{\psi}^{Y})\mathcal{H}.
\]
From $\partial_{\psi}\mathcal{H}=-Q_{x}$ as well as from $\partial_{\psi}^{X}\mathcal{H}=M_{x}\,,\,\partial_{\psi}^{Y}\mathcal{H}=M_{xy}$
we easily observe that a part of the variational derivate $\delta_{\psi}=\partial_{\psi}-d_{X}\partial_{\psi}^{X}-d_{Y}\partial_{\psi}^{Y}$
is incorporated in $\mathcal{J}_{SD}$ (namely $(1,-\partial_{X},-\partial_{Y})$)
and that the variables $Q_{x},\, M_{x}$ and $M_{xy}$ are used directly.

\subsection{The energy balances}

From the relation (\ref{eq:HdotPdeoO}) the power balance is derived
easily once the state $x$ as well as the Hamiltonian density $\mathcal{H}$
is chosen - this is very simple, since $\mathcal{J}$ is no differential
operator and does not contribute to the boundary term. This is different
in the approach as in section (\ref{sub:Approach-based-on}) since
the boundary operator $\mathcal{B}_{\mathcal{J}_{SD}}$ has to be
derived from the special choice of $\mathcal{J}_{SD}$ and by an additional
integration by parts one ends up again by the same relation as in
(\ref{eq:HdotMinGeo}).

\section{Conclusion}

We have presented two different port-Hamiltonian representations based
on a given set of partial differential equations together with their
boundary conditions derived by the Lagrangian formalism using jet-bundles.
By means of the running example, the Mindlin plate, we have extensively
discussed and compared these two different Hamiltonian scenarios.
Further investigations should also include the field theoretic Hamiltonian
concepts coming from mathematical physics, like the polysymplectic
and/or the multisymplectic approach as in \cite{Giachetta,Gotay}
or in the spirit as in \cite{Nishida}.


\bibliographystyle{IEEEtran}
\bibliography{IEEEabrv,maxibib}





\end{document}